\documentclass{amsart}

\usepackage{amsmath}
\usepackage{amssymb}
\usepackage{amsthm}

\DeclareMathOperator{\Id}{Id}

\DeclareMathOperator{\tr}{tr}

\DeclareMathOperator{\Vol}{Vol}
\DeclareMathOperator{\dvol}{dvol}

\DeclareMathOperator{\Ric}{Ric}

\DeclareMathOperator{\Conf}{Conf}

\newcommand{\od}{\overline{d}}

\newcommand{\ou}{\overline{u}}

\newcommand{\odelta}{\overline{\delta}}

\newcommand{\oDelta}{\overline{\Delta}}

\newcommand{\onabla}{\overline{\nabla}}

\newcommand{\lp}{\langle}
\newcommand{\rp}{\rangle}
\newcommand{\lv}{\lvert}
\newcommand{\rv}{\rvert}
\newcommand{\lV}{\lVert}
\newcommand{\rV}{\rVert}

\newcommand{\mC}{\mathcal{C}}

\newcommand{\mE}{\mathcal{E}}
\newcommand{\mF}{\mathcal{F}}
\newcommand{\mG}{\mathcal{G}}

\newcommand{\mS}{\mathcal{S}}

\newcommand{\mV}{\mathcal{V}}

\newcommand{\bN}{\mathbb{N}}

\newcommand{\bR}{\mathbb{R}}

\newcommand{\suchthat}{\mathrel{}\middle|\mathrel{}}

\def\sideremark#1{\ifvmode\leavevmode\fi\vadjust{\vbox to0pt{\vss
 \hbox to 0pt{\hskip\hsize\hskip1em
 \vbox{\hsize3cm\tiny\raggedright\pretolerance10000
 \noindent #1\hfill}\hss}\vbox to8pt{\vfil}\vss}}}

\newcommand{\comment}[1]{}

\newtheorem{thm}{Theorem}[section]
\newtheorem{prop}[thm]{Proposition}
\newtheorem{lem}[thm]{Lemma}
\newtheorem{cor}[thm]{Corollary}

\theoremstyle{definition}
\newtheorem{defn}[thm]{Definition}

\newtheorem{conj}[thm]{Conjecture}

\theoremstyle{remark}
\newtheorem{remark}[thm]{Remark}

\numberwithin{equation}{section}
\usepackage{mathtools}

\begin{document}

\title{On a fully nonlinear sharp Sobolev trace inequality}
\author{Jeffrey S. Case}
\thanks{JSC was partially supported by a grant from the Simons Foundation (Grant No.\ 524601)}
\address{109 McAllister Building \\ Penn State University \\ University Park, PA 16802}
\email{jscase@psu.edu}
\author{Yi Wang}
\thanks{YW was partially supported by NSF CAREER Award \ DMS-1845033}
\address{Department of Mathematics \\Johns Hopkins University\\ Baltimore, MD 21218}
\email{ywang@math.jhu.edu}
\keywords{conformally covariant operator; boundary operator; $\sigma_k$-curvature; Sobolev trace inequality; fully nonlinear PDE}
\subjclass[2010]{Primary 58J32; Secondary 53C21, 35J66, 58E11}
\begin{abstract}
 We classify local minimizers of $\int\sigma_2+\oint H_2$ among all conformally flat metrics in the Euclidean $(n+1)$-ball, $4\leq n\leq 5$, for which the boundary has unit volume, subject to an ellipticity assumption.
 We also classify local minimizers of the analogous functional in the critical dimension $n+1=4$.
 If minimizers exist, this implies a fully nonlinear sharp Sobolev trace inequality.  Our proof is an adaptation of the Frank--Lieb proof of the sharp Sobolev inequality, and in particular does not rely on symmetrization or Obata-type arguments.
\end{abstract}
\maketitle

\section{Introduction}
\label{sec:intro}

The first sharp Sobolev trace inequality was proven by Escobar~\cite{Escobar1988}.  In geometric terms, he showed that if $g=u^2dx^2$ is any conformally flat metric on the Euclidean ball $B^{n+1}\subset\bR^{n+1}$, $n>1$, of radius one, then
\begin{equation}
 \label{eqn:escobar}
 \frac{1}{2n}\int_{B^{n+1}} R^g\,\dvol_g + \oint_{S^n} H^g\,\dvol_{\iota^\ast g} \geq \omega_n^{\frac{1}{n}} \Vol_{\iota^\ast g}(S^n)^{\frac{n-1}{n}} ,
\end{equation}
where $\omega_n$ is the volume of the standard $n$-sphere, $\iota\colon S^n\to B^{n+1}$ is the inclusion of $S^n=\partial B^{n+1}$, and $H^g$ is the mean curvature of $S^n$ induced by $g$, with the convention that $S^n$ has mean curvature $1$ with respect to the standard metric.  Moreover, he showed that equality holds in~\eqref{eqn:escobar} if and only if $g$ is flat.  His proof relies on an Obata-type argument 
which classifies all scalar flat metrics $g=u^2dx^2$ on the ball for which the boundary has constant mean curvature.  The inequality~\eqref{eqn:escobar} plays a crucial role in studying a version of the boundary Yamabe problem; see~\cite{Almaraz2010,Escobar1992a,Marques2005,Marques2007,MayerNdiaye2017} and references therein.

In analytic terms, Equation~\eqref{eqn:escobar} states that
\begin{equation}
 \label{eqn:analytic_escobar}
 \int_{B^{n+1}} u\,L_2u + \oint_{S^n} u\,B_1u \geq \frac{n-1}{2}\omega_n^{\frac{1}{n}}\left(\oint_{S^n} \lv u\rv^{\frac{2n}{n-1}}\right)^{\frac{n-1}{n}}
\end{equation}
for all $u\in W^{1,2}(B^{n+1})$, where $L_2=-\Delta$ is the conformal Laplacian, $B_1=\partial_r+\frac{n-1}{2}$ is the conformal Robin operator~\cite{Cherrier1984,Escobar1988}, and all integrals are taken with respect to the Riemannian volume element of the Euclidean metric on $B^{n+1}$ or the induced metric on $S^n$, as appropriate.  Moreover, equality holds if~\eqref{eqn:analytic_escobar} if and only if
\[ u(x) = a\left| rx-\xi_0 \right|^{1-n} \]
for constants $a\in\bR$ and $r\in[0,1)$ and a point $\xi_0\in S^n$.  The inequalities~\eqref{eqn:escobar} and~\eqref{eqn:analytic_escobar} are equivalent due to the conformal covariafnce of $L_2$ and $B_1$.  Other proofs of~\eqref{eqn:analytic_escobar} which exploit conformal covariance and the linearity of $L_2$ and $B_1$ are known; e.g.\ \cite{Beckner1993,Case2015b}.

Given $k\in\bN$, Viaclovsky~\cite{Viaclovsky2000} defined the \emph{$\sigma_k$-curvature} of a Riemannian manifold $(X^{n+1},g)$ as the $k$-th elementary symmetric function of the eigenvalues of the \emph{Schouten tensor} $P:=\frac{1}{n-1}\bigl(\Ric - \frac{R}{2n}g\bigr)$.  For example, $\sigma_1=\frac{1}{2n}R$.  When written in terms of a fixed background metric $g$, the equation $\sigma_k^{u^2g}=f$ is a second-order fully nonlinear PDE which is elliptic in the positive $k$-cone; i.e.\ it is elliptic if $\sigma_j^{u^2g}>0$ for $1\leq j\leq k$; see~\cite{Viaclovsky2000}.  On closed manifolds, the equation $\sigma_k^{u^2g}=1$ is variational if and only if $k\leq 2$ or $g$ is locally conformally flat~\cite{BransonGover2008}.

Initial studies of the $\sigma_k$-curvature involved constructing minimizers of the total $\sigma_k$-curvature functional among all volume-normalized metrics in the positive $k$-cone (e.g.\ \cite{GurskyViaclovsky2003,LiLi2005,ShengTrudingerWang2007}).  In the critical case of dimension four, Chang, Gursky and Yang~\cite{ChangGurskyYang2002} noted that one could instead work in the positive $1$-cone provided the total $\sigma_2$-curvature was positive.   Later studies (e.g.\ \cite{GeWang2013,GuanLinWang2004,Sheng2008}) generalized this to show that one can minimize in the positive $(k-1)$-cone under a suitable integral assumption.  For example, combining results of Guan and Wang~\cite{GuanWang2004} and Ge and Wang~\cite{GeWang2013} yields sharp fully nonlinear Sobolev inequalities of closed $n$-spheres, $n>4$, stated in terms of the $\sigma_2$-curvature and the positive $1$-cone.  Note that Obata's argument generalizes to prove that any conformally flat metric of constant $\sigma_k$-curvature on the sphere has constant sectional curvature, subject to the above ellipticity condition~\cite{ChangGurskyYang2003b,Viaclovsky2000}.

Given $k\in\bN$, S.\ Chen~\cite{Chen2009s} defined the \emph{$H_k$-curvature} of the boundary of a Riemannian manifold $(X^{n+1},g)$ in terms of elementary symmetric functions of the Schouten tensor of the interior and the second fundamental form of the boundary.  The key points are that $H_1$ is the mean curvature, $H_k^{u^2g}$ depends only on the tangential two-jet of $u$ and the normal derivative of $u$ along the boundary, and, provided $k\leq2$ or $g$ is locally conformally flat,
\begin{equation}
 \label{eqn:primitive}
 \left.\frac{d}{dt}\right|_{t=0}\mS_k(e^{2t\Upsilon}g) = (n+1-2k)\left[ \int_X \sigma_k^g\Upsilon\,\dvol_g + \oint_{\partial X} H_k^g\Upsilon\,\dvol_{\iota^\ast g}\right]
\end{equation}
for any Riemannian manifold $(X^{n+1},g)$, $n+1\neq 2k$, and any $\Upsilon\in C^\infty(X)$, where
\[ \mS_k(g) := \int_X \sigma_k^g\,\dvol_g + \oint_{\partial X} H_k^g\,\dvol_{\iota^\ast g} . \]
In particular, $\mS_1(g)$ expresses the left-hand side of~\eqref{eqn:escobar}.
In the critical dimension $n=2k-1$, the conformal primitive $\mF_k$ of $(\sigma_k; H_k)$ is
\begin{equation}
 \label{eqn:primitivecritical}\mF_k(g_u):= \int_0^1\left\{ \int_X u\sigma_k^{g_s}\,\dvol_{g_s} + \oint_{\partial X} uH_k^{g_s}\,\dvol_{\iota^\ast g_s} \right\} ds \end{equation}
for all $u\in C^\infty(X)$, where $g$ is a fixed background metric, $g_u:=e^{2u}g$ and $g_s:=e^{2su}g$; see~\cite{CaseWang2016s}.

It follows from~\eqref{eqn:primitive} that the critical metrics of $\mS_k$ under the volume constraint $\Vol(\partial X)=1$ satisfy $\sigma_k^g=0$ in $X$ and have $H_k^g$ constant on $\partial X$; while the critical metrics of $\mS_k$ under the volume constraint $\Vol(X)=1$ have $\sigma_k^g$ constant in $X$ and $H_k^g=0$ on $\partial X$.

In light of the aforementioned results of Escobar, Ge--Wang, and Guan--Wang, one expects the following fully nonlinear sharp Sobolev trace inequality:

\begin{conj}
 \label{conj:sharp_trace}
 Let $(B^{n+1},dx^2)$ be the unit ball in Euclidean $(n+1)$-space and let $k<\frac{n+1}{2}$ be a positive integer.  For any metric $g\in\mC_{k-1}$,
 \begin{equation}
  \label{eqn:mCk-1}
  \mC_{k-1} := \left\{ g=u^2dx^2 \suchthat \sigma_j^g\geq0, H_j^g>0, 1\leq j\leq k-1 \right\} ,
 \end{equation}
 it holds that
 \[ \mS_k(g) \geq \frac{n!}{(n+1-k)!(2k-1)!!}\omega_n^{\frac{2k-1}{n}}\left(\Vol_{\iota^\ast g}(S^n)\right)^{\frac{n+1-2k}{n}} \]
 with equality if and only if $g$ is flat.
\end{conj}

Note that if $(X^{n+1},g)$ has umbilic boundary and $g\in\Gamma_k^+$, then $H_k^g>0$ if and only if $H>0$; see~\cite{CaseWang2016s}.



In the critical dimension $n+1=2k$, one instead expects a Lebedev--Milin-type inequality stated in terms of the functional $\mF_k$ (cf.\ \cite{AcheChang2017,Beckner1993,OsgoodPhillipsSarnak1}).  This is analogous to sharp Onofri-type inequalities known on closed spheres (cf.\ \cite{ChangYang1995,ChangYang2003,GuanWang2004,Onofri1982,OsgoodPhillipsSarnak1}).

\begin{conj}
 \label{conj:sharp_trace3}
 Let $(B^{n+1},dx^2)$ be the unit ball in Euclidean $(n+1)$-space and let $k=\frac{n+1}{2}$.  For any metric $g\in\mC_{k-1}$,
 it holds that
 \[ \mF_k(g) \geq \frac{2^{(n+1)/n}}{n(n+1)}\omega_n\log \frac{\Vol_{\iota^\ast g}(S^n)}{\omega_n} , \]
with equality if and only if $g$ is flat.
\end{conj}

As already noted, Conjecture~\ref{conj:sharp_trace} holds when $k=1$; Osgood, Phillips and Sarnak \cite{OsgoodPhillipsSarnak1} proved Conjecture \ref{conj:sharp_trace3} 
when $k=1$. 
For $k\geq2$, the authors~\cite{CaseWang2016s} showed that if $g=u^2dx^2$ is a $\sigma_k$-flat metric on $B^{n+1}$ for which $g\in\overline{\Gamma_k^+}$ and $\partial B^{n+1}$ has constant positive $H_k$-curvature, then, under a pinching condition on the mean curvature of $\partial B^{n+1}$, the metric $g$ is flat.  This was proven by adapting Escobar's Obata-type argument proving~\eqref{eqn:escobar}.  It is not clear how to remove the pinching assumption and the existence of a minimizer has not yet been studied.

The purpose of this note is to give further evidence for Conjecture~\ref{conj:sharp_trace} and Conjecture~\ref{conj:sharp_trace3} by removing the aforementioned pinching condition, at least in low dimensions. Define
\begin{equation}
 \label{eqn:mV_metric}
 \mV := \left\{ g= u^2 dx^2 \suchthat \Vol_{\iota^\ast g}(S^n) = \omega_n \right\} .
\end{equation}

\begin{thm}
 \label{thm:main_thm}
 Let $(B^{n+1},dx^2)$, $n=4,5$, be the unit ball in Euclidean $(n+1)$-space and suppose that $g\in \mC_1$ is a local minimizer of $\mS_2\colon \mV\to \mathbb R$.
 Then $g=dx^2$ up to the action of the conformal group of $B^{n+1}$.
\end{thm}

In comparison with our previous work~\cite{CaseWang2016s}, Theorem~\ref{thm:main_thm} removes the pinching assumption but imposes the stronger assumption that $g\in \mC_1$ is a local minimizer of $\mS_2\colon\mV\to\bR$, rather than just a critical point.  We expect that the dimension requirement $n\leq 5$ can be removed.

\begin{thm}
 \label{thm:main_thm3}
 Let $(B^{4},dx^2)$ be the unit ball in Euclidean four-space.
 Suppose that $g\in\mC_1$ is a local minimizer of $\mF_2\colon\mV\to\bR$. 
 Then $g=dx^2$ up to the action of the conformal group of $B^{4}$. 
\end{thm}

We remark that in Theorem~\ref{thm:main_thm} and Theorem~\ref{thm:main_thm3} we assume $g$ is a local minimizer of $\mV$ as our proofs are based on the first and the second variation formulas. We also require $g\in \mC_1$ for ellipticity. We do not know whether a local minimizer of $\mS_2$ (or $\mF_2$) on $\mV\cap \mC_1$ is a local minimizer on $\mV$, but hope to investigate this later.

We prove Theorem~\ref{thm:main_thm} and Theorem~\ref{thm:main_thm3} by adapting the rearrangement-free proof by Frank and Lieb~\cite{FrankLieb2012b} of Aubin's sharp Sobolev inequality~\cite{Aubin1976s}.  Indeed, this same technique gives a new proof of~\eqref{eqn:escobar}; see Subection~\ref{subsec:k1} for details.  To the best of our knowledge, this is the first time the Frank--Lieb argument has been employed on manifolds with boundary.

The Frank--Lieb argument exploits conformal covariance and a nice formula for the commutator of the conformal Laplacian on the sphere with a first spherical harmonic; similar properties allow Frank and Lieb to also prove sharp Sobolev inequalities on the CR spheres~\cite{FrankLieb2012a}.  Our proof also exploits conformal invariance and nice commutator formulae, this time both in the interior and on the boundary of $B^{n+1}$.  An intriguing question is whether our proofs can be adapted to CR manifolds.

This article is organized as follows.  In Section~\ref{sec:bg} we collect some useful background information on the $\sigma_2$- and $H_2$-curvatures.  In Section~\ref{sec:nonsharp} we give further evidence for Conjecture~\ref{conj:sharp_trace} and Conjecture~\ref{conj:sharp_trace3} by establishing non-sharp Sobolev trace and Lebedev--Milin-type inequalities when $k=2$.  In Section~\ref{sec:uniqueness} we explain how conformal invariance and the assumption of a local minimizer are used in the Frank--Lieb argument. In Section~\ref{sec:classification} we give a new proof of \eqref{eqn:escobar} and prove Theorem~\ref{thm:main_thm} and Theorem~\ref{thm:main_thm3}.


\section{Background}
\label{sec:bg}

Let $(X^{n+1},g)$ be a Riemannian manifold.  The \emph{Schouten tensor} is
\[ P = \frac{1}{n-1}\left(\Ric - \frac{R}{2n}g\right), \]
and its trace is $J=\frac{R}{2n}$.  Given $k\in\bN$, the \emph{$\sigma_k$-curvature} is the $k$-th elementary symmetric function of the eigenvalues of the Schouten tensor.  Alternatively,
\begin{align*}
 \sigma_1 & := J, \\
 \sigma_2 & := \frac{1}{2}\left( J^2 - \lv P\rv^2\right) .
\end{align*}
The \emph{first Newton tensor $T_1$} is the section of $S^2T^\ast X$ given by
\[ T_1 = Jg - P . \]
A consequence of G{\aa}rding's work on hyperbolic polynomials~\cite{Garding1959} is that if $g$ is in the positive elliptic $2$-cone,
\[ \Gamma_2^+ := \left\{ g \suchthat \sigma_1, \sigma_2 > 0 \right\} , \]
then $T_1>0$; see~\cite{CaffarelliNirenbergSpruck1985}.  Moreover, if
\[ g \in \overline{\Gamma_2^+} := \left\{ g \suchthat \sigma_1,\sigma_2 \geq 0 \right\} , \]
then $T_1\geq0$.  The importance of this observation comes from the conformal transformation formula for the $\sigma_2$-curvature:
\[ \left.\frac{\partial}{\partial t}\right|_{t=0} \sigma_2^{e^{2t\Upsilon}g} = -4\Upsilon\sigma_2^g - \lp T_1^g, \nabla_g^2\Upsilon\rp_g \]
for all metrics $g$ and all $\Upsilon\in C^\infty(X)$.  Note two facts: First, when restricted to a conformal class, the equation $\sigma_2^g=f$ is elliptic (resp.\ degenerate elliptic) 
when $g\in\Gamma_2^+$ (resp.\ $g\in\overline{\Gamma_2^+}$).  Second, $T_1$ is divergence-free~\cite{Viaclovsky2000}, and hence
\begin{equation}
 \label{eqn:sigmak_var}
 \left.\frac{\partial}{\partial t}\right|_{t=0} \sigma_2^{e^{2t\Upsilon}g} = -2k\Upsilon\sigma_2^g - \delta^g\left(T_1^g(\nabla^g\Upsilon)\right) .
\end{equation}

Suppose now that $(X^{n+1},g)$ has umbilic boundary $M^n:=\partial X$.  Let $\eta$ denote the outward-pointing unit normal along $M$ and let $H=\frac{1}{n}\tr_{\iota^\ast g}\nabla\eta$, where $\nabla\eta$ is regarded as a section of $S^2T^\ast M$.  The \emph{$H_2$-curvature}
\begin{equation}
 \label{def:H2umbilic}
 H_2 := H\tr_{\iota^\ast g} P\rv_{TM} + \frac{n}{3}H^3 .
\end{equation}
A key property of $H_2$ is its conformal linearization~\cite{CaseWang2016s,Chen2009s}:
\begin{equation}
 \label{eqn:Hk_var}
 \left.\frac{\partial}{\partial t}\right|_{t=0} H_2^{e^{2t\Upsilon}g} = -3\Upsilon H_2^g + T_1^g(\eta,\nabla\Upsilon) - \odelta\left(H^g\,\od\Upsilon\right),
\end{equation}
where $\od$ and $\odelta$ denote the intrinsic exterior derivative and divergence, respectively. 
The variational formula~\eqref{eqn:primitive}, which identifies $\frac{1}{n-3}\mS_2$ as a conformal primitive for $(\sigma_2;H_2)$ when $n+1\not=4$, follows immediately from~\eqref{eqn:sigmak_var} and~\eqref{eqn:Hk_var}. The identification of $\mF_2$ as the conformal primitive of
$(\sigma_2; H_2)$ when $n+1=4$ likewise follows immediately from \eqref{eqn:sigmak_var} and \eqref{eqn:Hk_var}; see~\cite{CaseWang2016s}.  See~\cite{CaseWang2016s,Chen2009s} for a discussion of analogous properties for manifolds with nonumbilic boundary.

The authors' previous work~\cite{CaseWang2016s} introduced two conformally covariant polydifferential operators which help to study the functional $\mS_2$.  Specifically, let $(X^{n+1},g)$ be a compact Riemannian manifold with boundary and suppose $n+1>4$.  Define
\begin{align*}
 L_4(u) & := \left(\frac{n-3}{4}u\right)^3 u^{\frac{16}{n-3}}\sigma_2^{g_u}, \\
 B_3(u) & := \left(\frac{n-3}{4}u\right)^3 u^{\frac{12}{n-3}}H_2^{g_u}
\end{align*}
for $g_u:=u^{\frac{8}{n-3}}g$.  Then $L_4$ and $B_3$ are both homogeneous polynomials of degree $3$ in the two-jet of $u$, and hence their polarizations define multilinear operators
\begin{align*}
 & L_4 \colon \left(C^\infty(X)\right)^3 \to C^\infty(X), \\
 & B_3 \colon \left(C^\infty(X)\right)^3 \to C^\infty(M)
\end{align*}
which are differential in each of their inputs.  These operators have two key properties.  First, they are conformally covariant: If $g_u=u^{\frac{8}{n-3}}g$, then
\begin{align*}
 L_4^{g_u}(w_1,w_2,w_3) & = u^{-\frac{3(n+1)+4}{n-3}}L_4^g(uw_1,uw_2,uw_3), \\
 B_3^{g_u}(w_1,w_2,w_3) & = (\iota^\ast u)^{-\frac{3(n+1)}{n-3}}B_3^g(uw_1,uw_2,uw_3) .
\end{align*}
Second, the pair $(L_4;B_3)$ is formally self-adjoint: The map
\[ (u_0,u_1,u_2,u_3) \mapsto \int_X u_0\,L_4(u_1,u_2,u_3) + \oint_{\partial X} \iota^\ast u_0\,B_3(u_1,u_2,u_3) \]
is symmetric on $\bigl(C^\infty(X)\bigr)^4$.

We require the following explicit formula for $L_4$ and $B_3$ under certain geometric conditions:

\begin{lem}
 \label{lem:L4}
 Let $(X^{n+1},g)$ be a Ricci flat manifold.  Then
 \begin{align*}
  L_4(u,u,u) & = \frac{1}{2}\delta\left(\lv\nabla u\rv^2\,du\right) - \frac{n-3}{16}\Bigl[ u\Delta\lv\nabla u\rv^2 - \delta\left((\Delta u^2)\,du\right) \Bigr] .
 \end{align*}
\end{lem}

\begin{proof}
 Since $g$ is Ricci flat, the Schouten tensor of $g_u:=u^{\frac{8}{n-3}}g$ is
 \begin{equation}
  \label{eqn:conformal_schouten}
  P^{g_u} = -\frac{4}{n-3}u^{-1}\nabla^2u + \frac{4(n+1)}{(n-3)^2}u^{-2}du\otimes du - \frac{8}{(n-3)^2}u^{-2}\lv\nabla u\rv^2g
 \end{equation}
 (cf.\ \cite[Equation~3.6]{CaseWang2016s}).  Therefore, regarding $\nabla^2u$ as a section of $T^\ast M\otimes TM$,
 \begin{align*}
  L_4(u) & = \frac{n-3}{4}u\sigma_2\left(-\nabla^2u + \frac{n+1}{n-3}u^{-1}du\otimes\nabla u - \frac{2}{n-3}u^{-1}\lv\nabla u\rv^2\Id\right) \\
  & = \frac{n-3}{8}u\biggl\{ \left(\Delta u + \frac{n+1}{n-3}u^{-1}\lv\nabla u\rv^2\right)^2 \\
   & \qquad - \left| \nabla^2u - \frac{n+1}{n-3}u^{-1}du\otimes du + \frac{2}{n-3}u^{-1}\lv\nabla u\rv^2g\right|^2 \biggr\} \\
  & = \frac{n-3}{8}u(\Delta u)^2 - \frac{n-3}{8}u\lv\nabla^2u\rv^2 + \frac{n-1}{4}\lv\nabla u\rv^2\Delta u + \frac{n+1}{4}\nabla^2u(\nabla u,\nabla u) \\
  & = -\frac{n-3}{16}u\Delta\lv\nabla u\rv^2 + \frac{n-3}{8}\delta\left(u(\Delta u)\,du\right) + \frac{n+1}{8}\delta\left(\lv\nabla u\rv^2\,du\right) \\
  & = \frac{1}{2}\delta\left(\lv\nabla u\rv^2\,du\right) - \frac{n-3}{16}u\Delta\lv\nabla u\rv^2 + \frac{n-3}{16}\delta\left((\Delta u^2)\,du\right),
 \end{align*}
 where the fourth equality also uses the assumption that $g$ is Ricci flat.
\end{proof}

\begin{lem}
 \label{lem:B4}
 Let $(X^{n+1},g)$ be a Ricci flat manifold with umbilic boundary of constant mean curvature $H$.  Then
 \begin{align*}
  B_3(u,u,u) & = -\frac{n}{6}\left(\eta u + \frac{n-3}{4}Hu\right)^3 \\
  & \quad + \left(\eta u + \frac{n-3}{4}Hu\right)\left(-\frac{n-3}{4}u\oDelta u - \frac{n-1}{4}\lv\onabla u\rv^2 + \frac{n(n-3)^2}{32}H^2u^2\right) .
 \end{align*}
\end{lem}

\begin{proof}
 On the one hand, the conformal transformation law for the mean curvature implies that
 \begin{equation}
  \label{eqn:conformal_mean_curvature}
  \frac{n-3}{4}u^{\frac{n+1}{n-3}}H^{g_u} = \eta u + \frac{n-3}{4}Hu .
 \end{equation}
 On the other hand, the assumptions that $g$ is Ricci flat and $\partial X$ is umbilic imply, using~\eqref{eqn:conformal_schouten}, that
 \begin{multline}
  \label{eqn:conformalschouten}
  \left(\frac{n-3}{4}\right)^2u^{\frac{2(n+1)}{n-3}}\tr_{\iota^\ast g_u} P^{g_u}\rv_{TM} \\ = -\frac{n-3}{4}u\oDelta u - \frac{n(n-3)}{4}Hu\eta u - \frac{n-1}{4}\lv\onabla u\rv^2 - \frac{n}{2}(\eta u)^2 .
 \end{multline}
 Combining these formulae with the definition of $B_3$ yields the desired result.
\end{proof}

It will be useful to express $L_4$ and $B_3$ in alternative forms.  To that end, we introduce some operators.

\begin{defn}
 \label{defn:curvature_operators}
 Let $(B^{n+1},dx^2)$ be the unit ball in Euclidean $(n+1)$-space.  We define $\sigma_1\colon C^\infty(B)\to C^\infty(B)$, $T_1\colon C^\infty(B)\to C^\infty(B;S^2T^\ast B)$, and $H\colon C^\infty(B)\to C^\infty(\partial B)$ by
 \begin{align*}
  \sigma_1(u) & := -\frac{n-3}{4}u\Delta u - \frac{n+1}{4}\lv\nabla u\rv^2, \\
  T_1(u) & := \left(\sigma_1(u)+\frac{1}{2}\lv\nabla u\rv^2\right)dx^2 + \frac{n-3}{4}u\nabla^2u - \frac{n+1}{4}du\otimes du, \\
  H(u) & := \eta u + \frac{n-3}{4}u .
 \end{align*}
\end{defn}

As suggested by our notation, the point of these operators is that they are closely related to the corresponding geometric objects defined with respect to the metric $g_u:=u^{\frac{8}{n-3}}dx^2$, but with the extra benefit of being polynomial in $u$ and its covariant derivatives.  The relations to geometric objects defined with respect to $g_u$ are given by the following lemma.  This also indicates how to extend the definitions of $\sigma_1$, $T_1$, and $H$ to general manifolds with boundary.

\begin{lem}
 \label{lem:curvature_operators}
 Let $(B^{n+1},dx^2)$ be the unit ball in Euclidean $(n+1)$-space.  Then
 \begin{align*}
  \sigma_1(u) & = \left(\frac{n-3}{4}\right)^2 u^{\frac{2(n+1)}{n-3}} \sigma_1^{g_u}, \\
  T_1(u) & = \left(\frac{n-3}{4}\right)^2 u^{2} T_1^{g_u}, \\
  H(u) & = \frac{n-3}{4} u^{\frac{n+1}{n-3}} H^{g_u} ,
 \end{align*}
 where $g_u:=u^{\frac{8}{n-3}}dx^2$.  In particular, each of $\sigma_1$, $T_1$, and $H$ is conformally covariant.
\end{lem}

\begin{proof}
 The equations for $\sigma_1(u)$ and $T_1(u)$ follow from~\eqref{eqn:conformal_schouten}.  The equation for $H(u)$ follows from~\eqref{eqn:conformal_mean_curvature}.
\end{proof}

A useful corollary of Lemma~\ref{lem:curvature_operators} is the following expression for $T_1(u)(\eta,\eta)$.

\begin{cor}
 \label{cor:curvature_operators}
 Let $(B^{n+1},dx^2)$ be the unit ball in Euclidean $(n+1)$-space.  Then
 \[ T_1(u)(\eta,\eta) = -\frac{n-3}{4}u\oDelta u - \frac{n-1}{4}\lv\onabla u\rv^2 - \frac{n}{2}(\eta u)^2 - \frac{n(n-3)}{4}u\eta u . \]
\end{cor}

\begin{proof}
 On the one hand, Lemma~\ref{lem:curvature_operators} implies that
 \[ T_1(u)(\eta,\eta) = \left(\frac{n-3}{4}\right)^2u^{\frac{2(n+1)}{n-3}} T_1^{g_u}(\eta^{g_u},\eta^{g_u}) . \]
 On the other hand, it holds that
 \begin{equation}
 \label{eqn:T1gu} T_1^{g_u}(\eta^{g_u},\eta^{g_u}) =\tr_{\iota^\ast g_u} P^{g_u}\rv_{TM} .
 \end{equation}
 Applying~\eqref{eqn:conformalschouten} yields the desired conclusion.
\end{proof}

Lemma~\ref{lem:curvature_operators} also implies the following useful formulas for $L_4$ and $B_3$.

\begin{prop}
 \label{prop:curvature_operators}
 Let $(B^{n+1},dx^2)$ be the unit ball in Euclidean $(n+1)$-space.  Then
 \begin{align*}
  uL_4(u,u,u) & = \left(\sigma_1(u) + \frac{1}{2}\lv\nabla u\rv^2\right)\lv\nabla u\rv^2 - \frac{1}{2}\delta\left(uT_1(u)(\nabla u) + \frac{1}{2}u\lv\nabla u\rv^2\,du\right) , \\
  B_3(u,u,u) & = H(u)T_1(u)(\eta,\eta) + \frac{n}{3}H(u)^3
 \end{align*}
 for all $u\in C^\infty(B)$.
\end{prop}

\begin{proof}
 The formula for $B_3(u,u,u)$ follows immediately from Lemma~\ref{lem:curvature_operators}, the definitions of $H_2$ and $B_3$, and the identity~\eqref{eqn:T1gu}.

 Recall that $\sigma_2^{g_u}=\frac{1}{2}\lp T_1^{g_u},P^{g_u}\rp_{g_u}$.  Using~\eqref{eqn:conformal_schouten}, Lemma~\ref{lem:curvature_operators} and the definition of $L_4$, we see that
 \[ \frac{n-3}{4}uL_4(u) = \frac{1}{2}\left\lp T_1(u), -\frac{n-3}{4}u\nabla^2u + \frac{n+1}{4}du\otimes du - \frac{1}{2}\lv\nabla u\rv^2dx^2\right\rp_{dx^2} . \]
 Using the fact that $\delta^{g_u} T_1^{g_u}=0$, we see that
 \begin{equation}
  \label{eqn:divT1}
  \delta\left(uT_1(u)\right) = -\frac{n+5}{n-3}T_1(u)(\nabla u) + \frac{4n}{n-3}\sigma_1(u)\,du .
 \end{equation}
 Combining this with the previous display yields
 \begin{equation}
  \label{eqn:step1}
  \frac{n-3}{4}uL_4(u) = -\frac{n-3}{8}\delta\left(uT_1(u)(\nabla u)\right) - \frac{1}{2}T_1(u)(\nabla u,\nabla u) + \frac{n}{4}\lv\nabla u\rv^2\sigma_1(u) .
 \end{equation}
 Now observe that
 \begin{equation}
  \label{eqn:4Laplacian}
  \begin{split}
  \frac{n-3}{4}\delta\left(u\lv\nabla u\rv^2\,du\right) & = \frac{n-3}{4}\left(u\lv\nabla u\rv^2\Delta u + 2u\nabla^2u(\nabla u,\nabla u) + \lv\nabla u\rv^4\right) \\
  & = 2T_1(u)(\nabla u,\nabla u) - 3\lv\nabla u\rv^2\sigma_1(u) + \frac{n-3}{2}\lv\nabla u\rv^4 .
  \end{split}
 \end{equation}
 Combining this with~\eqref{eqn:step1} yields the formula for $uL_4(u,u,u)$.
\end{proof}

\subsection{The four-dimensional case}
\label{subsec:bg/4d}

In dimension four, the behavior of $\sigma_2$ and $H_2$ under conformal change of metric is also controlled by conformally covariant polydifferential operators.  The following result can also be derived from Lemma~\ref{lem:L4} and Lemma~\ref{lem:B4} by analytic continuation in the dimension.  
See~\cite{CaseLinYuan2018b} for a general discussion on closed manifolds.

\begin{lem}
 \label{lem:critical_operators}
 Let $(X^4,g)$ be a Riemannian manifold with umbilic boundary.  Define operators $L_{4,j}\colon\bigl(C^\infty(X)\bigr)^j\to C^\infty(X)$, $j=1,2,3$, by
 \begin{align*}
  L_{4,3}(u,v,w) & = \delta\left(\lp\nabla u,\nabla v\rp\,dw + \lp\nabla u,\nabla w\rp\,dv + \lp\nabla v,\nabla w\rp\,du\right) , \\
  L_{4,2}(u,v) & = -\frac{1}{2}\left(\Delta\lp\nabla u,\nabla v\rp - \delta\left((\Delta u)\,dv + (\Delta v)\,du\right)\right),\\
  L_{4,1}(u) & = -\delta\left(T_1(\nabla u)\right) ,
 \end{align*}
 and operators $B_{3,j}\colon \bigl(C^\infty(X)\bigr)^j\to C^\infty(M)$ , $j=1,2,3$, by
 \begin{align*}
  B_{3,3}(u,v,w) & = -\left(\lp\nabla u,\nabla v\rp\,\eta w + \lp\nabla u,\nabla w\rp\,\eta v + \lp\nabla v,\nabla w\rp\,\eta u\right) , \\
  B_{3,2}(u,v) & = -\left((\oDelta u)\eta v + (\oDelta v)\eta u\right) - H\lp\onabla u,\onabla v\rp - 3H(\eta u)(\eta v), \\
  B_{3,1}(u) & = T_1(\eta,\eta)\eta u - H\oDelta u .
 \end{align*}
 Then
 \begin{align}
  \label{eqn:critical_sigma2} e^{4u}\sigma_2^{e^{2u}g} & = \sigma_2^g + L_{4,1}(u) + \frac{1}{2}L_{4,2}(u,u) + \frac{1}{6}L_{4,3}(u,u,u), \\
  \label{eqn:critical_H2} e^{3u}H_2^{e^{2u}g} &^ = H_2^g + B_{3,1}(u) + \frac{1}{2}B_{3,2}(u,u) + \frac{1}{6}B_{3,3}(u,u,u)
 \end{align}
 for all $u\in C^\infty(X)$.
\end{lem}

\begin{proof}
 We directly compute that
 \[ \left.\frac{\partial^j}{\partial t^j}\right|_{t=0} e^{4tu}\sigma_2^{e^{2tu}g} = L_{4,j}(\underbrace{u,\dotsc,u}_{\text{$j$ times}}) \]
 for all integers $1\leq j\leq 4$, with the convention $L_{4,4}=0$.  Integrating along the path $t\mapsto e^{2tu}g$, $t\in[0,1]$, yields~\eqref{eqn:critical_sigma2}.

 Since $\partial X$ is umbilic, we directly compute that
 \[ \left.\frac{\partial^j}{\partial t^j}\right|_{t=0} e^{3tu}H_2^{e^{2tu}g} = B_{3,j}(\underbrace{u,\dotsc,u}_{\text{$j$ times}}) \]
 for all integers $1\leq j\leq 4$, with the convention $B_{3,4}=0$.
 Integrating along the path $t\mapsto e^{2tu}g$, $t\in[0,1]$, yields~\eqref{eqn:critical_H2}.
\end{proof}

One important property of the operators $L_{4,j}$ and $B_{3,j}$ is their transformation under conformal change of metrics, generalizing~\eqref{eqn:critical_sigma2} and~\eqref{eqn:critical_H2}, respectively.

\begin{cor}
 \label{cor:critical_transformation}
 Let $(X^4,g)$ be a Riemannian manifold with umbilic boundary.  For any integer $1\leq j\leq 3$, it holds that
 \begin{align}
  \label{eqn:critical_transformation_L4} e^{4\Upsilon}L_{4,j}^{e^{2\Upsilon}g}(u_1,\dotsc,u_j) & = \sum_{\ell=j}^{3} \frac{1}{(\ell-j)!}L_{4,\ell}^g\bigl(u_1,\dotsc,u_j,\underbrace{\Upsilon,\dotsc,\Upsilon}_{\text{$\ell-j$ times}}\bigr), \\
  \label{eqn:critical_transformation_B3} e^{3\Upsilon}B_{3,j}^{e^{2\Upsilon}g}(u_1,\dotsc,u_j) & = \sum_{\ell=j}^{3} \frac{1}{(\ell-j)!}B_{3,\ell}^g\bigl(u_1,\dotsc,u_j,\underbrace{\Upsilon,\dotsc,\Upsilon}_{\text{$\ell-j$ times}}\bigr)
 \end{align}
 for all $\Upsilon,u_1,\dotsc,u_j\in C^\infty(X)$, where $L_{4,j}^g$ and $B_{3,j}^g$ (resp.\ $L_{4,j}^{e^{2\Upsilon}g}$, $B_{3,j}^{e^{2\Upsilon}g}$) are defined with respect to the metric $g$ (resp.\ the metric $e^{2\Upsilon}g$).
\end{cor}

\begin{remark}
 One can easily show that Lemma~\ref{lem:critical_operators} and Corollary~\ref{cor:critical_transformation} also hold in the nonumbilic case with only a slight change to the definition of $B_{3,1}$.
\end{remark}

\begin{proof}
 Using~\eqref{eqn:critical_sigma2} to compute $e^{4(\Upsilon+tu)}\sigma_2^{e^{2(\Upsilon+tu)}g}$ in two ways yields
 \begin{align*}
  \MoveEqLeft e^{4\Upsilon}\left[ \sigma_2^{e^{2\Upsilon}g} + tL_{4,1}^{e^{2\Upsilon}g}(u) + \frac{t^2}{2}L_{4,2}^{e^{2\Upsilon}g}(u,u) + \frac{t^3}{6}L_{4,3}^{e^{2\Upsilon}g}(u,u,u) \right] \\
  & = \sigma_2^g + L_{4,1}^g(\Upsilon+tu) + \frac{1}{2}L_{4,2}^g(\Upsilon+tu,\Upsilon+tu) \\
   & \qquad + \frac{1}{6}L_{4,3}(\Upsilon+tu,\Upsilon+tu,\Upsilon+tu) .
 \end{align*}
 Equating coefficients of $t$ and polarizing yields~\eqref{eqn:critical_transformation_L4}.  The verification of~\eqref{eqn:critical_transformation_B3} follows similarly from~\eqref{eqn:critical_H2}.
\end{proof}

Another important property of the operators $L_{4,j}$ and $B_{3,j}$ is that the pairs $(L_{3,j};B_{3,j})$ are formally self-adjoint; i.e.\ the maps
\[ (u_0,\dotsc,u_j) \mapsto \int_{X^4} u_0\,L_{4,j}(u_1,\dotsc,u_j) + \oint_{\partial X} \iota^\ast u_0\,B_{3,j}(u_1,\dotsc,u_j) \]
are symmetric on $\bigl(C^\infty(X)\bigr)^{j+1}$ for all integers $1\leq j\leq 3$.  This is a consequence of the following three computational lemmas.

\begin{lem}
 \label{lem:critical_fsa1}
 Let $(X^4,g)$ be a compact Riemannian manifold with umbilic boundary.  Then
 \[ \int_X u\,L_{4,1}(v) + \oint_{\partial X} \iota^\ast u\,B_{3,1}(v) = \int_X T_1(\nabla u,\nabla v) + \oint_{\partial X} H\lp\onabla u,\onabla v \rp \]
 for all $u,v\in C^\infty(X)$.
\end{lem}

\begin{proof}
 It directly follows from Lemma~\ref{lem:critical_operators} that
 \[ \int_X u\,L_{4,1}(v) + \oint_{\partial X} \iota^\ast u\,B_{3,1}(v) = \int_X T_1(\nabla u,\nabla v) - \oint_{\partial X} \left[ u T_1(\eta,\onabla v) + Hu\oDelta v \right] . \]
 Since $\partial X$ is umbilic, $T_1(\eta,\onabla v) = \lp\onabla H,\onabla v\rp$ (see~\cite[Lemma~2.1]{Case2015b}).  The conclusion readily follows.
\end{proof}

\begin{lem}
 \label{lem:critical_fsa2}
 Let $(X^4,g)$ be a compact Riemannian manifold with umbilic boundary.  Then
 \begin{align*}
  \MoveEqLeft \int_X u\,L_{4,2}(v,w) + \oint_{\partial X} \iota^\ast u\,B_{3,2}(v,w) \\
   & = -\frac{1}{2}\int_X \left[ \lp\nabla u,\nabla v\rp\,\Delta w + \lp\nabla u,\nabla w\rp\,\Delta v + \lp\nabla v,\nabla w\rp\,\Delta u \right] \\
   & \quad + \frac{1}{2}\oint_{\partial X} \left[ \lp\onabla u,\onabla v\rp\,\eta w + \lp\onabla u,\onabla w\rp\,\eta v + \lp\onabla v,\onabla w\rp\,\eta u + (\eta u)(\eta v)(\eta w) \right]
 \end{align*}
 for all $u,v,w\in C^\infty(X)$.
\end{lem}

\begin{proof}
 It follows directly from Lemma~\ref{lem:critical_operators} that
 \begin{align*}
  \MoveEqLeft \int_X u\,L_{4,2}(v,w) + \oint_{\partial X} \iota^\ast u\,B_{3,2}(v,w) \\
  & = -\frac{1}{2}\int_X \left[ \lp\nabla u,\nabla v\rp\,\Delta w + \lp\nabla u,\nabla w\rp\,\Delta v + \lp\nabla v,\nabla w\rp\,\Delta u \right] \\
  & \quad - \frac{1}{2}\oint_{\partial X} \Bigl[ u\eta\lp\nabla v,\nabla w\rp - \lp\nabla v,\nabla w\rp\,\eta u - u(\eta v)\Delta w - u(\eta w)\Delta v \\
   & \qquad\qquad\qquad + 2u(\eta v)\oDelta w + 2u(\eta w)\oDelta v + 2Hu\lp\onabla v,\onabla w\rp + 6Hu(\eta v)(\eta w) \Bigr]
 \end{align*}
 Since $\partial X$ is umbilic,
 \begin{multline*}
  \eta\lp\nabla v,\nabla w\rp - (\eta v)\Delta w - (\eta w)\Delta v = \lp\onabla v,\onabla\eta w\rp + \lp\onabla w,\onabla\eta v\rp - 2H\lp\onabla v,\onabla w\rp \\ - (\eta v)\oDelta w - (\eta w)\oDelta v - 6H(\eta v)(\eta w)
 \end{multline*}
 (see~\cite[Lemma~2.3]{Case2015b}).  The conclusion readily follows.
\end{proof}

\begin{lem}
 \label{lem:critical_fsa3}
 Let $(X^4,g)$ be a compact Riemannian manifold with umbilic boundary.  Then
 \begin{multline*}
  \int_X t\,L_{4,3}(u,v,w) + \oint_{\partial X} \iota^\ast t\,B_{3,3}(u,v,w) \\ = -\int_X \Bigl[ \lp\nabla t,\nabla u\rp\lp\nabla v,\nabla w\rp + \lp\nabla t,\nabla v\rp\lp\nabla u,\nabla w\rp + \lp\nabla t,\nabla w\rp\lp\nabla u,\nabla v\rp \Bigr]
 \end{multline*}
 for all $t,u,v,w\in C^\infty(X)$.
\end{lem}

\begin{proof}
 This follows directly from Lemma~\ref{lem:critical_operators}.
\end{proof}

\section{A non-sharp fully nonlinear Sobolev trace inequality}
\label{sec:nonsharp}

The remainder of this article is concerned with the functional
\[ \mE_2(u) := \int_B u\,L_4(u,u,u) + \oint_{\partial B} u\,B_3(u,u,u) \]
defined with respect to the unit ball in Euclidean $(n+1)$-space and its analogue when $n=3$.  Note that, by the conformal invariance of $L_4$ and $B_3$,
\begin{equation}
 \label{eqn:conformal_covariance_mE2}
 \mE_2(u) = \left(\frac{n-3}{4}\right)^3\mS_2\left(u^{\frac{8}{n-3}}dx^2\right)
\end{equation}
for all positive $u\in C^\infty(B)$.  The main result of this section is the following (non-sharp) fully nonlinear Sobolev trace inequality in
\[ \mC_1 := \left\{ u\in C^\infty(B) \suchthat \sigma_1(u)\geq0, H(u)>0 \right\} . \]
Note that $\mC_1$ equals the set~\eqref{eqn:mCk-1} under the correspondence $u\cong u^{\frac{8}{n-3}}dx^2$.
 

\begin{thm}
 \label{thm:nonsharp}
 Let $(B^{n+1},dx^2)$ be the unit ball in Euclidean $(n+1)$-space.  Then
 \[ \inf \left\{ \mE_2(u) \suchthat u\in\mC_1, \oint_{\partial B}\lv u\rv^{\frac{4n}{n-3}}=1 \right\} > 0 . \]
\end{thm}

\begin{proof}
 We first derive a general formula for $\mE_4(u)$ making no assumptions on $u$.  Proposition~\ref{prop:curvature_operators} implies that
 \begin{multline}
  \label{eqn:mE4_pregeometric}
  \mE_2(u) = \int_B \left( \sigma_1(u) + \frac{1}{2}\lv\nabla u\rv^2\right)\lv\nabla u\rv^2 \\ + \oint_{\partial B} \left[ uH(u)T_1(u)(\eta,\eta) - \frac{1}{2}uT_1(u)(\nabla u,\eta) - \frac{1}{4}u\lv\nabla u\rv^2\eta u + \frac{n}{3}uH(u)^3\right].
 \end{multline}
 To simplify this, first note that
 \begin{align*}
  T_1(u)(\nabla u,\eta) & = T_1(u)(\onabla u,\eta) + T_1(u)(\eta,\eta)\eta u, \\
  T_1(u)(\onabla u,\eta) & = \frac{n-3}{4}u\lp\onabla u,\onabla\eta u\rp - \frac{n-3}{4}u\lv\onabla u\rv^2 - \frac{n+1}{4}\lv\onabla u\rv^2\eta u ,
 \end{align*}
 where the second equality uses Definition~\ref{defn:curvature_operators} and the fact that $\partial B$ is umbilic with second fundamental form $II=\iota^\ast dx^2$.  We conclude that
 \begin{align*}
  \oint_{\partial B} uT_1(u)(\nabla u,\eta) & = \oint_{\partial B} \biggl[ uT_1(u)(\eta,\eta)\eta u - \frac{n-3}{4}u^2(\eta u)\oDelta u \\
   & \qquad - \frac{3n-5}{4}u\lv\onabla u\rv^2\eta u - \frac{n-3}{4}u^2\lv\onabla u\rv^2\biggr] \\
  & = \oint_{\partial B} \biggl[ 2uT_1(u)(\eta,\eta)\eta u - \frac{n-2}{2}u\lv\onabla u\rv^2\eta u + \frac{n}{2}u(\eta u)^3 \\
   & \qquad + \frac{n(n-3)}{4}u^2(\eta u)^2 - \frac{n-3}{4}u^2\lv\onabla u\rv^2 \biggr] ,
 \end{align*}
 where the second equality uses Corollary~\ref{cor:curvature_operators}.  Combining this with~\eqref{eqn:mE4_pregeometric} and using the definition of $H(u)$ yields
 \begin{multline}
  \label{eqn:mE4_geometric}
  \mE_2(u) = \int_B \left( \sigma_1(u) + \frac{1}{2}\lv\nabla u\rv^2\right)\lv\nabla u\rv^2 \\ + \oint_{\partial B} \biggl[ \frac{n-3}{4}u^2T_1(u)(\eta,\eta) + \frac{n-3}{4}uH(u)\lv\onabla u\rv^2 - \frac{(n-3)(n-5)}{16}u^2\lv\onabla u\rv^2 \\+ \frac{n-3}{12}u(\eta u)^3 + \frac{n(n-3)}{8}u^2(\eta u)^2 + \frac{n(n-3)^2}{16}u^3\eta u + \frac{n}{3}\left(\frac{n-3}{4}\right)^3u^4 \biggr].
 \end{multline}
 Using Corollary~\ref{cor:curvature_operators} again yields
 \begin{equation}
  \label{eqn:int_T1eta}
  \oint_{\partial B} u^2T_1(u)(\eta,\eta) = \oint_{\partial B} \left[ \frac{n-4}{2}u^2\lv\onabla u\rv^2 - \frac{n}{2}u^2(\eta u)^2 - \frac{n(n-3)}{4}u^3\eta u \right] .
 \end{equation}
 Combining this with~\eqref{eqn:mE4_geometric} yields
 \begin{multline}
  \label{eqn:mE4_positive}
  \mE_2(u) = \int_B \left( \sigma_1(u) + \frac{1}{2}\lv\nabla u\rv^2\right)\lv\nabla u\rv^2 + \oint_{\partial B} \biggl[ \frac{n-3}{4}uH(u)\lv\onabla u\rv^2 \\+ \left(\frac{n-3}{4}\right)^2u^2\lv\onabla u\rv^2 + \frac{n-3}{12}u(\eta u)^3 + \frac{n}{3}\left(\frac{n-3}{4}\right)^3u^4 \biggr] .
 \end{multline}

 Now suppose that $u\in\mC_1$.  Since $H(u)>0$, it holds that $(\eta u)^3>-\bigl(\frac{n-3}{4}\bigr)^3u^3$.  It then follows from~\eqref{eqn:mE4_positive} that
 \[ \mE_2(u) > \frac{1}{2}\int_B \lv\nabla u\rv^4 + \left(\frac{n-3}{4}\right)^2\oint_{\partial B} \left[ u^2\lv\onabla u\rv^2 + \frac{(n+1)(n-3)}{16}u^4 \right] . \]
 The conclusion follows from the Sobolev trace embedding $W^{1,4}(B)\hookrightarrow L^{\frac{4n}{n-3}}(\partial B)$ and the existence of a constant $C$ such that $\int_B u^4 \leq C (\int_B |\nabla u|^4+ \oint_{\partial B} u^4 )$.
\end{proof}

\subsection{The four dimensional case.}
\label{subsec:nonsharp/4d}

In dimension four, the critical dimension for $k=2$, one instead expects a Lebedev--Milin-type inequality (cf.\ \cite{AcheChang2017,Beckner1993,ChangYang1995,OsgoodPhillipsSarnak1}) involving the conformal primitive $\mF^{g}_2$ of $(\sigma_2;H_2)$ given in \eqref{eqn:primitivecritical}.
As in the case of noncritical dimension, given a four-dimensional Riemannian manifold $(X^4,g)$ with umbilic boundary, it is convenient to define the functional $\mF_2\colon C^\infty(X)\to\bR$ by
\begin{equation}
 \label{eqn:mF2_functional}
 \mF_2^g(u) := \mF_2^g(e^{2u}g) .
\end{equation}
We omit the superscript $g$ when the background metric is clear from context.  We require the following equivalent formula for $\mF_2$.

\begin{lem}
 \label{lem:critical_functional}
 Let $(X^4,g)$ be a compact Riemannian manifold with umbilic boundary. 
Then
 \begin{multline*}
  \mF_2(u) = \int_X \left[ \frac{1}{24}uL_{4,3}(u,u,u) + \frac{1}{6}uL_{4,2}(u,u) + \frac{1}{2}L_{4,1}(u) + \sigma_2^gu \right]\,\dvol_g \\ + \oint_{\partial X} \left[ \frac{1}{24}uB_{3,3}(u,u,u) + \frac{1}{6}uB_{3,2}(u,u) + \frac{1}{2}uB_{3,1}(u) + H_2^gu\right]\,\dvol_{\iota^\ast g}.
 \end{multline*}
\end{lem}

\begin{proof}
 Lemma~\ref{lem:critical_operators} immediately implies that
 \begin{multline*}
  \mF_2(u) = \int_0^1\Biggl\{ \int_X \left[ \frac{s^3}{6}uL_{4,3}(u,u,u) + \frac{s^2}{2}uL_{4,2}(u,u) + suL_{4,1}(u) + \sigma_1^gu \right]\,\dvol_g \\ + \oint_{\partial X} \left[ \frac{s^3}{6}uB_{3,3}(u,u,u) + \frac{s^2}{2}uB_{3,2}(u,u) + suB_{3,1}(u) + H_2^gu \right]\,\dvol_{\iota^\ast g} \Biggr\}\,ds .
 \end{multline*}
 The conclusion readily follows.
\end{proof}

The functional $\mF_2$ is conformally invariant, in the sense that it satisfies the following cocycle condition (cf.\ \cite{BransonGover2007}).

\begin{lem}
 \label{lem:critical_mF_invariance}
 Let $(X^4,g)$ be a compact Riemannian manifold with umbilic boundary.  Then
 \begin{equation}
  \label{eqn:critical_mF_cohomology}
  \mF_2^g(u+v) = \mF_2^g(v) + \mF_2^{e^{2v}g}(u)
 \end{equation}
 for all $u,v\in C^\infty(X)$.
\end{lem}

\begin{proof}
 This follows directly from Corollary~\ref{cor:critical_transformation} and Lemma~\ref{lem:critical_functional}.
\end{proof}

Adapting an argument of Chang and Yang~\cite{ChangYang1995} yields the following specialization of Lemma~\ref{lem:critical_mF_invariance} to the Euclidean four-ball.

\begin{cor}
 \label{cor:critical_mF_invariance}
 Let $(B^4,dx^2)$ be the unit ball in Euclidean four-space.  Then
 \[ \mF_2(u) = \mF_2\left(\Phi^\ast u + \frac{1}{4}\log\lv J_\Phi\rv\right) \]
 for all $u\in C^\infty(B)$ and all $\Phi\in\Conf(B^4; S^3)$, the group of conformal diffeomorphisms of $B^4$ which fix the boundary, where $\lv J_\Phi\rv$ is the Jacobian determinant, $\dvol_{\Phi^\ast dx^2} = \lv J_\Phi\rv\,\dvol_{dx^2}$.
\end{cor}

\begin{proof}
 First observe that~\eqref{eqn:critical_mF_cohomology} yields
 \begin{equation}
  \label{eqn:cy_step1}
  \mF_2^{dx^2}(u) = \mF_2^{\Phi^\ast dx^2}(\Phi^\ast u) = \mF_2^{dx^2}\left(\Phi^\ast u + \frac{1}{4}\log\lv J_\Phi\rv\right) - \mF_2^{dx^2}\left(\frac{1}{4}\log\lv J_\Phi\rv\right) .
 \end{equation}

 Next let $t\mapsto\Upsilon_t\in\Conf(B^4; S^3)$ be a one-parameter family of conformal diffeomorphisms of $B^4$ with $\Phi_0=\Id$ and $\Phi_1=\Phi$ and set $\Upsilon_t=\frac{1}{4}\log\lv J_{\Phi_t}\rv$.  In particular, $\Upsilon_0=0$, and hence $\mF_2(\Upsilon_0)=0$.  Using Lemma~\ref{lem:critical_fsa1}, Lemma~\ref{lem:critical_fsa2} and Lemma~\ref{lem:critical_fsa3}, we compute that
 \begin{multline}
  \label{eqn:cy_deriv}
  \frac{d}{dt}\mF_2^{dx^2}(\Upsilon_t) = \frac{1}{6}\int_{B^4} \dot\Upsilon_t \left[ L_{4,3}(\Upsilon_t,\Upsilon_t,\Upsilon_t) + 3L_{4,2}(\Upsilon_t,\Upsilon_t)\right] \\ + \frac{1}{6}\oint_{S^3} \dot\Upsilon_t \left[ B_{3,3}(\Upsilon_t,\Upsilon_t,\Upsilon_t) + 3B_{3,2}(\Upsilon_t,\Upsilon_t) + 6B_{3,1}(\Upsilon_t) + 6\right] ,
 \end{multline}
 where $\dot\Upsilon_t:=\frac{\partial\Upsilon_t}{\partial t}$.  Lemma~\ref{lem:critical_operators} implies that
 \begin{align*}
  0 & = 6\lv J_{\Phi_t}\rv\sigma_2^{\Phi_t^\ast dx^2} = L_{4,3}(\Upsilon_t,\Upsilon_t,\Upsilon_t) + 3L_{4,2}(\Upsilon_t,\Upsilon_t) , \\
  6\iota^\ast\lv J_{\Phi_t}\rv & = 6\lv J_{\Phi_t}\rv H_2^{\Phi_t^\ast dx^2} = B_{3,3}(\Upsilon_t,\Upsilon_t,\Upsilon_t) + 3B_{3,2}(\Upsilon_t,\Upsilon_t) + 6B_{3,1}(\Upsilon_t) + 6
 \end{align*}
 for all $t\in[0,1]$.  Inserting this into~\eqref{eqn:cy_deriv} yields
 \[ \frac{d}{dt}\mF_2^{dx^2}(\Upsilon_t) = \oint_{S^3} \dot\Upsilon_t\,\dvol_{\iota^\ast\Phi_t^\ast dx^2} = 0 , \]
 where the last equality uses the fact that $\Vol_{\Phi_t^\ast dx^2}(S^3)$ is constant.  In particular,
 \[ \mF_2\left(\frac{1}{4}\log\lv J_\Phi\rv\right) = \mF_2(\Upsilon_1) = 0 . \]
 Inserting this into~\eqref{eqn:cy_step1} yields the desired conclusion.
\end{proof}

It is more useful to write $\mF_2(u)$ after integration by parts.  Given our focus in this article,
we restrict our attention to the unit ball in Euclidean four-space.

\begin{lem}
 \label{lem:critical_functional_weak}
 Let $(B^4,dx^2)$ be the unit ball in Euclidean four-space.  Then
 \begin{multline*}
  \mF_2(u) = -\frac{1}{4}\int_{B^4} \left[ \frac{1}{2}\lv\nabla u\rv^4 + \lv\nabla u\rv^2\Delta u \right] \\ + \oint_{S^3} \left[ \frac{1}{4}\lv\onabla u\rv^2\eta u + \frac{1}{12}(\eta u)^3 + \frac{1}{2}\lv\onabla u\rv^2 + u \right]
 \end{multline*}
 for all $u\in C^\infty(B)$.
\end{lem}

\begin{proof}
 First observe that $dx^2$ is Ricci flat and $S^3=\partial B^4$ is umbilic and has constant mean curvature $H=1$.  The conclusion follows from Lemma~\ref{lem:critical_fsa1}, Lemma~\ref{lem:critical_fsa2}, Lemma~\ref{lem:critical_fsa3} and Lemma~\ref{lem:critical_functional}.
\end{proof}

To establish our non-sharp Lebedev--Milin-type inequality, we again need to restrict to the conformal metrics of nonnegative scalar curvature and positive mean curvature.  To that end, define $\sigma_1\colon C^\infty(B^4)\to C^\infty(B^4)$ and $H\colon C^\infty(B^4)\to C^\infty(S^3)$ by
\begin{align*}
 \sigma_1(u) & := -\Delta u - \lv\nabla u\rv^2 , \\
 H(u) & := \eta u + 1 .
\end{align*}
Note that
\begin{align*}
 e^{2u}\sigma_1^{e^{2u}dx^2} & = \sigma_1(u), \\
 e^uH^{e^{2u}dx^2} & = H(u) ,
\end{align*}
justifying our notation.  Set
\[ \mC_1 := \left\{ u\in C^\infty(B^4) \suchthat \sigma_1(u)\geq0, H(u)>0 \right\} . \]
Note that $\mC_1$ equals the set~\eqref{eqn:mCk-1} under the correspondence $u\cong e^{2u}dx^2$ used in~\eqref{eqn:mF2_functional}.  Our non-sharp Lebedev--Milin-type inequality establishes a uniform lower bound on the functional 
$\mG_2\colon\mC_1\to\bR$,
\begin{equation}
 \label{eqn:mG}
 \mG_2(u) := \mF_2(u) - \frac{\omega_3}{3}\log\frac{\oint e^{3u}}{\oint 1} .
\end{equation}

\begin{thm}
 \label{thm:critical_nonsharp}
 Let $(B^4,dx^2)$ be the unit ball in Euclidean four-space.
  Then
 \[ \inf \left\{ \mG_2(u) \suchthat u\in\mC_1 \right\} > -\infty.  \]
\end{thm}

\begin{remark}
 As noted in the Conjecture \ref{conj:sharp_trace3} of Section \ref{sec:intro}, the expected sharp Lebedev--Milin-type inequality is that the infimum equals zero.  For general Riemannian four-manifolds with boundary, the quantity to consider is
 \[ \mF_2(e^{2u}g) - \frac{1}{3}\left[ \int_X\sigma_2^g\,\dvol_g + \oint_{\partial X} H_2^g\,\dvol_{\iota^\ast g}\right]\log\oint_{\partial X} e^{3u}\,\dvol_{\iota^\ast g}, \]
 as this functional is scale invariant and its critical points are those conformal metrics with $\sigma_2=0$ and $H_2$ constant.
\end{remark}

\begin{proof}
 It follows from Lemma~\ref{lem:critical_functional_weak} that
 \begin{multline*}
  \mF_2(u) = \frac{1}{4}\int_{B^4} \left(\sigma_1(u) + \frac{1}{2}\lv\nabla u\rv^2\right)\lv\nabla u\rv^2 \\ + \oint_{S^3} \left[ \frac{1}{4}H(u)\lv\onabla u\rv^2 + \frac{1}{12}(\eta u)^3 + \frac{1}{4}\lv\onabla u\rv^2 + u \right]
 \end{multline*}
 for all $u\in C^\infty(B^4)$.  In particular, if $u\in \mC_1$, then
 \[ \mG_2(u) > -\frac{1}{12}\omega_3 + \frac{1}{4}\oint_{S^3} \lv\onabla u\rv^2 - \frac{\omega_3}{3}\log\frac{\oint e^{3(u-\ou)}}{\oint 1} \]
 where $\ou:=\omega_3^{-1}\oint u$ is the average of $u$.  The conclusion follows from the Poincar\'e inequality and Jensen's inequality.
\end{proof} 
\section{A spectral inequality at local minimizers}
\label{sec:uniqueness}

The Frank--Lieb argument~\cite{FrankLieb2012b} proving sharp Hardy--Littlewood--Sobolev inequalities begins with a spectral inequality satisfied by any local minimizer of the problem in question.  When $n\geq4$, we are concerned with the local minimizers of $\mE_2\colon\mV\to\bR$ for
\[ \mV := \left\{ u\in C^\infty(B^{n+1}) \suchthat u>0, \oint_{\partial B} u^{\frac{4n}{n-3}} = \omega_n \right\} . \]
Note that $\mV$ equals the set~\eqref{eqn:mV_metric} under the correspondence $u\cong u^{\frac{8}{n-3}}dx^2$.  The spectral inequality satisfied by such local minimizers is as follows:

\begin{prop}
 \label{prop:second_variation}
 Let $(B^{n+1},dx^2)$ be the unit ball in Euclidean $(n+1)$-space and let $u$ be a local minimizer of $\mE_2\colon\mV\to\bR$.  Then
 \[ \int_B uv\,L_4(uv,u,u) + \oint_{\partial B} uv\,B_3(uv,u,u) \geq \frac{n+1}{n-3}\omega_n^{-1}\mE_2(u)\oint_{\partial B} v^2u^{\frac{4n}{n-3}} \]
 for all $v\in C^\infty(B)$ such that
 \begin{equation}
  \label{eqn:tangent}
  \oint_{\partial B} vu^{\frac{4n}{n-3}} = 0 .
 \end{equation}
\end{prop}

\begin{proof}
 Set $u_t=\omega_n^{(n-3)/4n}\lV (1+tv)u\rV^{-1} (1+tv)u$, where $\lV\cdot\rV$ denotes the $L^{\frac{4n}{n-3}}(\partial B)$-norm of the restriction to $\partial B$.  Since $v$ satisfies~\eqref{eqn:tangent}, we see that $u_t$ is a curve in $\mV$ with $u_0=u$ and $\left.\frac{\partial}{\partial t}\right|_{t=0}u_t=uv$.  Note that
 \begin{equation}
  \label{eqn:second_variation_volume}
  \left.\frac{\partial^2}{\partial t^2}\right|_{t=0}u_t = -\frac{3(n+1)}{n-3}\omega_n^{-1}\left(\oint_{\partial B} v^2u^{\frac{4n}{n-3}}\right)u .
 \end{equation}

 Now, since $u$ is a critical  point of $\mE_2\colon\mV\to\bR$, it satisfies
 \begin{equation}
  \label{eqn:euler_equation}
  \begin{split}
  L_4(u,u,u) & = 0, \\
  B_4(u,u,u) & = \omega_n^{-1}\mE_2(u)u^{\frac{3(n+1)}{n-3}} .
  \end{split}
 \end{equation}
 Combining this with~\eqref{eqn:second_variation_volume} yields
 \begin{multline*}
  \frac{1}{12}\left.\frac{d^2}{dt^2}\right|_{t=0}\mE_2(u_t) \\ = \int_B uv\,L_4(uv,u,u) + \oint_{\partial B} uv\,B_3(uv,u,u) - \frac{n+1}{n-3}\omega_n^{-1}\mE_2(u)\oint_{\partial B} v^2u^{\frac{4n}{n-3}} .
 \end{multline*}
 Now apply the assumption that $u$ is a local minimizer of $\mE_2\colon\mV\to\bR$.
\end{proof}

Define the commutators of $L_4$ and $B_3$ with multiplication operators by
\begin{align*}
 [L_4,x](u,u,u) & := L_4(xu,u,u) - xL_4(u,u,u) , \\
 [B_3,x](u,u,u) & := B_3(xu,u,u) - xB_3(u,u,u) .
\end{align*}
The core of the Frank--Lieb argument is contained in the following estimate.

\begin{cor}
 \label{cor:balanced_minimizer}
 Let $(B^{n+1},dx^2)$ be the unit ball in Euclidean $(n+1)$-space and let $u$ be a local minimizer of $\mE_2\colon\mV\to\bR$.  Suppose additionally that
 \begin{equation}
  \label{eqn:balanced}
  \oint_{\partial B} x^iu^{\frac{4n}{n-3}} = 0
 \end{equation}
 for all $i\in\{1,\dotsc,n+1\}$, where $x^1,\dotsc,x^{n+1}$ denote the standard Cartesian coordinates.  Then
 \begin{equation}
  \label{eqn:balanced_spectral_estimate}
  \sum_{i=1}^{n+1} \left\{ \int_B ux^i\,[L_4,x^i](u,u,u) + \oint_{\partial B} ux^i\,[B_3,x^i](u,u,u) \right\} \geq \frac{4}{n-3}\mE_2(u) .
 \end{equation}
\end{cor}

\begin{proof}
 Since $u$ satisfies~\eqref{eqn:balanced}, the functions $v=x^i$ all satisfy~\eqref{eqn:tangent}.  Proposition~\ref{prop:second_variation} then implies that
 \[ \sum_{i=1}^{n+1} \left\{ \int_B ux^i\,L_4(x^iu,u,u) + \oint_{\partial B} ux^i\,B_3(ux^i,u,u) \right\} \geq \frac{n+1}{n-3}\mE_2(u) . \]
 The conclusion now follows from~\eqref{eqn:euler_equation} and the definitions of the commutators.
\end{proof}

One typically calls functions $u$ which satisfy~\eqref{eqn:balanced} \emph{balanced}. 
It is well-known \cite{ChangYang1987,FrankLieb2012b} that this condition can always be achieved by a suitable M\"obius transformation.  For conformally covariant problems, this means that local minimizers, if they exist, can always be taken to be balanced. Specifically:

\begin{prop}
 \label{prop:balanced}
 Let $(B^{n+1},dx^2)$ be the unit ball in Euclidean $(n+1)$-space and let $u$ be a local minimizer of $\mE_2\colon\mV\to\bR$.  Then there is a $\Phi\in\Conf(B^{n+1};S^n)$ such that
 \[ u_\Phi := \lv J_\Phi\rv^{\frac{n-3}{4(n+1)}}u\circ\Phi \]
 is a balanced local minimizer of $\mE_2\colon\mV\to\bR$.
\end{prop}

\begin{proof}
 First observe that $u\in\mV$ if and only if $u_\Phi\in\mV$ by change of variables. Moreover, by the diffeomorphism invariance of $\mE_2$, we see that
 \begin{equation}\label{eqn:E2} \mE_2(u_\Phi)=\mE_2(u).\end{equation}
 for all $u\in C^\infty(B)$ and all $\Phi\in\Conf(B^{n+1};S^n)$.
 In fact, Equation~\eqref{eqn:conformal_covariance_mE2} implies that
 \[ \left(\frac{4}{n-3}\right)^3\mE_2(u_\Phi)= \int_{B} \sigma_2^{g_{u_\Phi}}\dvol_{g_{u_\Phi}} + \oint_{\partial B} H_2^{g_{u_\Phi}}\dvol_{\iota^\ast g_{u_\Phi}} , \]
 where $g_{u_\Phi}=u_{\Phi}^{\frac{8}{n-3}}g= (u\circ\Phi)^{\frac{8}{n-3}} \lv J_\Phi\rv^{\frac{2}{n+1}} g$. By the fact that $\Phi$ is a conformal diffeomorphism of $B^{n+1}$, it holds that $\lv J_\Phi\rv^{\frac{2}{n+1}} g= \Phi^\ast g$ and thus $g_{u_\Phi}= \Phi^\ast g_u$. The curvatures $\sigma_2^{g_u}$ and $H_2^{g_u}$ are both invariant under diffeomorphism.  This yields~\eqref{eqn:E2}.
 
 Next, applying~\cite[Lemma~B.1]{FrankLieb2012b} to the restriction of the function $\lv u\rv^{\frac{4n}{n-3}}$ to $\partial B^{n+1}$ yields an element $\widetilde{\Phi}\in\Conf(S^n)$ such that
 \[ \lv J_{\widetilde{\Phi}}\rv^{\frac{n-3}{4n}} u \circ \widetilde{\Phi}(x)  \]
 satisfies~\eqref{eqn:balanced}. This is because
   \[
   \oint_{\partial B}  x^{i} \lv J_{\widetilde{\Phi}}(x)\rv  \lv u\circ \widetilde{\Phi}(x)\rv^{\frac{4n}{n-3}}dx= \oint_{\partial B}  (\widetilde{\Phi}^{-1}(y) )^{i} \lv u (y)\rv^{\frac{4n}{n-3}}dy=0,
   \]
 where we take $\widetilde{\Phi}^{-1}$ to be the conformal transformation $\gamma_{\delta, \xi}$ of $S^{n}$ in \cite[Lemma~B.1]{FrankLieb2012b}. Let $\Phi$ be the (unique) extension of $\widetilde{\Phi}$ in $\Conf(B^{n+1};S^n)$. This yields the desired conclusion.
\end{proof}

\subsection{The four-dimensional case}
\label{subsec:uniqueness/4d}

When $n=3$, the relevant functional is~\eqref{eqn:mG}.  This functional is scale invariant (i.e.\ $\mG_2(u+c)=\mG_2(u)$ for all $c\in\bR$), so there is no need to impose an additional volume normalization.  The analogue of Proposition~\ref{prop:second_variation} in this case is as follows:

\begin{prop}
 \label{prop:critical_second_variation}
 Let $(B^4,dx^2)$ be the unit ball in Euclidean four-space and let $u$ be a local minimizer of $\mG_2\colon C^\infty(B)\to\bR$.  Then
 \begin{multline*}
  3\omega_3\frac{\oint v^3e^{3u}}{\oint e^{3u}} \leq \frac{1}{2}\int_{B^4} vL_{4,3}(v,u,u) + \int_{B^4} vL_{4,2}(v,u) \\ + \frac{1}{2}\oint_{S^3} v\,B_{3,3}(v,u,u) + \oint_{S^3} v\,B_{3,2}(v,u) + \oint_{S^3} v\,B_{3,1}(v)
 \end{multline*}
 for all $v\in C^\infty(B)$ such that
 \begin{equation}
  \label{eqn:critical_tangent}
  \oint_{S^3} ve^{3u} = 0 .
 \end{equation}
\end{prop}

\begin{proof}
 It follows from Lemma~\ref{lem:critical_fsa1}, Lemma~\ref{lem:critical_fsa2} and Lemma~\ref{lem:critical_fsa3} that
 \begin{multline*}
  \left.\frac{d^2}{dt^2}\right|_{t=0}\mG_2(u+tv) = \frac{1}{2}\int_{B^4} vL_{4,3}(v,u,u) + \int_{B^4} v\,L_{4,2}(v,u) + \frac{1}{2}\oint_{S^3} vB_{3,3}(v,u,u) \\ + \oint_{S^3} vB_{3,2}(v,u) + \oint_{S^3} vB_{3,1}(v) - 3\omega_3\left[\frac{\oint v^2e^{3u}}{\oint e^{3u}} - \left(\frac{\oint ve^{3u}}{\oint e^{3u}}\right)^2 \right]
 \end{multline*}
 for all $u,v\in C^\infty(B)$.  The conclusion follows by assuming that $u$ is a local minimizer of $\mG_2\colon C^\infty(B)\to\bR$ and $v$ satisfies~\eqref{eqn:critical_tangent}.
\end{proof}

Note that the operators $L_{4,j}$, $B_{3,j}$, $j=1,2,3$, annihilate constants, in the sense that they give the zero function if at least one of their inputs is constant.  For this reason the following immediate consequence of Proposition~\ref{prop:critical_second_variation} for balanced minimizers is sufficient.

\begin{cor}
 \label{cor:critical_balanced_minimizer}
 Let $(B^{4},dx^2)$ be the unit ball in Euclidean four-space and let $u$ be a local minimizer of $\mG_2\colon C^\infty(B)\to\bR$.  Suppose additionally that
 \begin{equation}
  \label{eqn:critical_balanced}
  \oint_{S^3} x^i e^{3u} = 0
 \end{equation}
 for all $i\in\{1,\dotsc,4\}$, where $x^1,\dotsc,x^4$ denote the standard Cartesian coordinates.  Then
 \begin{multline*}
  3\omega_3 \leq \sum_{i=1}^4 \Biggl\{ \frac{1}{2}\int_{B^4} x^i\,L_{4,3}(x^i,u,u) + \int_{B^4} x^i\,L_{4,2}(x^i,u) \\ + \frac{1}{2}\oint_{S^3} x^iB_{3,3}(x^i,u,u) + \oint_{S^3} x^i\,B_{3,2}(x^i,u) + \oint_{S^3} x^i\,B_{3,1}(x^i) .
 \end{multline*}
\end{cor}

As in the higher-dimensional case, one can always assume that a local minimizer of $\mG_2\colon C^\infty(B)\to\bR$ satisfies the balancing condition~\eqref{eqn:critical_balanced} (cf.\ Proposition~\ref{prop:balanced}).

\begin{prop}
 \label{prop:critical_balanced}
 Let $(B^{4},dx^2)$ be the unit ball in Euclidean four-space and let $u$ be a local minimizer of $\mG_2\colon C^\infty(B)\to\bR$.  Then there is a $\Phi\in\Conf(B^4;S^3)$ such that
 \[ u_\Phi := u\circ\Phi + \frac{1}{4}\log\lv J_\Phi\rv \]
 is a balanced local minimizer of $\mG_2\colon C^\infty(B)\to\bR$.
\end{prop}

\begin{proof}
 First observe that
 \[ \oint_{S^3} e^{3u_\Phi}\dvol_{\iota^\ast dx^2} = \oint_{S^3} e^{3u}\,\dvol_{\iota^\ast dx^2} . \]
 Combining this with Corollary~\ref{cor:critical_mF_invariance} yields
 \[ \mG_2(u_\Phi) = \mG_2(u) \]
 for all $u\in C^\infty(B)$ and all $\Phi\in\Conf(B^4;S^3)$.  Applying~\cite[Lemma~B.1]{FrankLieb2012a} to the function $e^{3u}\circ\iota$ on $\partial B$ yields an element $\widetilde{\Phi}\in\Conf(S^3)$ such that
 \[ u\circ\iota\circ\widetilde{\Phi} + \frac{1}{3}\log\lv J_{\widetilde{\Phi}}\rv \]
 satisfies~\eqref{eqn:critical_balanced}.  Let $\Phi$ be the (unique) extension of $\widetilde{\Phi}$ in $\Conf(B^{n+1};S^n)$. This yields the desired conclusion.
\end{proof} 
\section{Classification of local minimizers}
\label{sec:classification}

As indicated in Section~\ref{sec:uniqueness}, it remains to compute the commutators $[L_4,x^i]$ and $[B_3,x^i]$.
To illustrate this strategy in a simple case, we first give a new proofs that the only local minimizers of Escobar's sharp Sobolev trace inequality~\eqref{eqn:analytic_escobar} are the constant functions and their images under the action of the conformal group.

\subsection{Escobar's functional}
\label{subsec:k1}

The analogue of Corollary~\ref{cor:balanced_minimizer} is that
\begin{equation}
 \label{eqn:k1_balanced_minimizer}
 \sum_{i=1}^{n+1} \left\{ \int_B ux^i[L_2,x^i](u) + \oint_{\partial B} ux^i[B_1,x^i](u) \right\} \geq \frac{2}{n-1}\mE_1(u)
\end{equation}
for all positive balanced local minimizers $u$ of
\[ \mE_1(u) := \int_B u\,L_2(u) + \oint_{\partial B} u\,B_1(u) \]
in the set
\[ \mV_1 := \left\{ 0<u\in C^\infty(B) \suchthat \oint_{\partial B} u^{\frac{2n}{n-1}} = \omega_n \right\}, \]
where a function $u$ is balanced if
\[ \oint_{\partial B} x^iu^{\frac{2n}{n-1}} = 0 \]
for all integers $1\leq i\leq n+1$.  Here we use the standard correspondence $u\cong u^{\frac{4}{n-1}}dx^2$ between functions on $B^{n+1}$ and conformally flat metrics.

It is straightforward to compute that
\begin{align*}
 [L_2,x^i](u) & = -2\lp\nabla x^i,\nabla u\rp, \\
 [B_1,x^i](u) & = x^iu .
\end{align*}
Inserting this into~\eqref{eqn:k1_balanced_minimizer} and using the formula
\[ \mE_1(u) = \int_B \lv\nabla u\rv^2 + \frac{n-1}{2}\oint_{\partial B}u^2 \]
yields
\begin{equation}
 \label{eqn:k1_apply_commutator}
 -\int_B u\lp\nabla u,\nabla r^2\rp \geq \frac{2}{n-1}\int_B\lv\nabla u\rv^2 ,
\end{equation}
where $r^2$ is the squared distance from the origin.  Since $u$ is a local minimizer of $\mE_1\colon\mV_1\to\bR$, it satisfies $L_2u=0$.  Integrating this against $(1-r^2)u$ yields
\[0\leq  \int_B (1-r^2)\lv\nabla u\rv^2 = \int_B u\lp\nabla u,\nabla r^2\rp . \]
Combining this with~\eqref{eqn:k1_apply_commutator} yields
\[ 0 \geq \int_B \left[ (1-r^2)\lv\nabla u\rv^2 + \frac{2}{n-1}\lv\nabla u\rv^2\right] . \]
Therefore $u$ is constant.

\subsection{The functional $\mE_2$}
\label{subsec:k2}

Our objective is to classify local minimizers of the functional
\[ \mE_2(u) = \int_B u\,L_4(u,u,u) + \oint_{\partial B} u\,B_3(u,u,u) \]
defined on the set
\[ \mV = \left\{ 0<u\in C^\infty(B^{n+1}) \suchthat \oint_{\partial B} u^{\frac{8n}{n-3}} = \omega_n \right\} . \]
We assume our minimizers are in the nonnegative cone
\[ \mC_1 = \left\{ u\in C^\infty(B^{n+1}) \suchthat \sigma_1(u)\geq 0, H(u) > 0 \right\} . \]
Note that local minimizers of $\mE_2\colon\mV\to\bR$ are such that the first variation vanishes and the second variation is nonnegative.


Our first task is to compute the commutators $[L_4,x^i]$ and $[B_3,x^i]$.  This is accomplished in the following two lemmas.

\begin{lem}
 \label{lem:L4_commutator}
 Let $B^{n+1}$ be the unit ball in $(n+1)$-dimensional Euclidean space and let $x$ denote a Cartesian coordinate in $\bR^{n+1}$.  Then
 \[ [L_4,x](u,u,u) = \frac{8}{3(n-3)}\left(T_1(\nabla u,\nabla x) - \frac{n}{2}\sigma_1\lp\nabla u,\nabla x\rp\right) . \]
\end{lem}

\begin{proof}
 It follows immediately from Lemma~\ref{lem:L4} that
 \begin{align*}
  [L_4,x](u,u,u) & = \frac{1}{6}\Bigl[ 2\delta\left(u\lp\nabla u,\nabla x\rp\,du\right) + \delta\left(u\lv\nabla u\rv^2\,dx\right) + 3\lv\nabla u\rv^2\lp\nabla u,\nabla x\rp\Bigr] \\
   & \quad - \frac{n-3}{48}\Bigl[ 2u\Delta\left(u\lp\nabla u,\nabla x\rp\right) + 4u\lp\nabla x,\nabla\lv\nabla u\rv^2\rp - 8\delta\left(u\lp\nabla u,\nabla x\rp\,du\right) \\
   & \qquad - \delta\left(u(\Delta u^2)\,dx\right) - 3\lp\nabla u,\nabla x\rp\,\Delta u^2 \Bigr] .
 \end{align*}
 Expanding this out yields
 \[ [L_4,x](u,u,u) = \frac{n-2}{3}u\lp\nabla u,\nabla x\rp\,\Delta u + \frac{n}{3}\lv\nabla u\rv^2\lp\nabla u,\nabla x\rp + \frac{1}{3}u\lp\nabla x,\nabla\lv\nabla u\rv^2\rp . \]
 Rewriting this using Definition~\ref{defn:curvature_operators} yields the desired result.
\end{proof}

\begin{lem}
 \label{lem:B3_commutator}
 Let $B^{n+1}$ be the unit ball in $(n+1)$-dimensional Euclidean space and let $x$ denote a Cartesian coordinate in $\bR^{n+1}$.  Then
 \[ [B_3,x](u,u,u) = \frac{xu}{3}\left(T_1(u)(\eta,\eta) + \frac{n(n-3)}{4}uH(u)\right) - \frac{n-2}{3}uH(u)\lp\onabla u,\onabla x\rp . \]
\end{lem}

\begin{proof}
 Recall that $\partial B$ is umbilic with constant mean curvature $1$.  It follows from Lemma~\ref{lem:B4} that
 \begin{align*}
  [B_4,x](u,u,u) & = -\frac{nxu}{6}\left(\eta u + \frac{n-3}{4}u\right)^2 \\
   & \quad + \frac{xu}{3}\left(-\frac{n-3}{4}u\oDelta u - \frac{n-1}{4}\lv\onabla u\rv^2 + \frac{n(n-3)^2}{32}u^2\right) \\
   & \quad + \frac{1}{3}\left(\eta u + \frac{n-3}{4}u\right)\left(\frac{n(n-3)}{4}xu^2 - \frac{n-2}{2}\lp\onabla u^2,\onabla x\rp\right) .
 \end{align*}
 The final conclusion follows from Definition~\ref{defn:curvature_operators} and Corollary~\ref{cor:curvature_operators}.
\end{proof}

Analogous to Subsection~\ref{subsec:k1}, the application of the commutator formula in Lemma~\ref{lem:L4_commutator} will produce an interior integral involving $T_1(\nabla u,\nabla r^2)$.  Our second task is to find a useful estimate for this integral.  

\begin{lem}
 \label{lem:integral_nablar}
 Let $(B^{n+1}, dx^2)$ be the unit ball in $(n+1)$-dimensional Euclidean space and let $r^2\in C^\infty(B)$ denote the squared-distance from the origin.  Then
 \begin{multline*}
  -\frac{4}{3(n-3)}\int_{B}u\left(T_1(u)-\frac{n}{2}\sigma_1(u)g\right)(\nabla u,\nabla r^2) = -\frac{n(n-5)}{6}\int_B u^2\lv\nabla u\rv^2 \\ + \oint_{\partial B} \left[ \frac{n-4}{6}u^2\lv\onabla u\rv^2 - \frac{n}{6}u^2(\eta u)^2 \right] .
 \end{multline*}
\end{lem}

\begin{proof}
 On the one hand, it follows from~\eqref{eqn:divT1} and the identities $\nabla r^2=2dx^2$ and $\tr T_1(u)=n\sigma_1(u)$ that
 \[ \delta\left(u^2T_1(u)(\nabla r^2)\right) = -\frac{8}{n-3}u\left(T_1(u) - \frac{n}{2}\sigma_1(u)g\right)(\nabla u,\nabla r^2) + 2nu^2\sigma_1(u) . \]
 On the other hand, it follows from Definition~\ref{defn:curvature_operators} that
 \[ u^2\sigma_1(u) = -\frac{n-3}{4}\delta\left(u^3\,du\right) + \frac{n-5}{2} u^2\lv\nabla u\rv^2 . \]
 Combining these displays yields
 \begin{multline*}
  -\frac{8}{n-3}\int_B u\left(T_1(u) - \frac{n}{2}\sigma_1(u)g\right)(\nabla u,\nabla r^2) = -n(n-5)\int_B u^2\lv\nabla u\rv^2 \\ + \oint_{\partial B} \left[ 2u^2T_1(u)(\eta,\eta) + \frac{n(n-3)}{2}u^3\eta u \right] .
 \end{multline*}
 The final conclusion follows from~\eqref{eqn:int_T1eta}.
\end{proof}


The final ingredient we need is an alternative formula for the energy $\mE_2(u)$ which does not include any terms of the form $\oint u(\eta u)^3$ and is manifestly positive in
\[ \mC_2 = \left\{ 0<u\in C^\infty(B) \suchthat \sigma_1^{g_u},\sigma_2^{g_u}\geq 0, H^{g_u}>0, g_u=u^{\frac{8}{n-3}}g \right\} . \]

\begin{lem}
 \label{lem:alternate_energy}
 Let $B^{n+1}$ be the unit ball in $(n+1)$-dimensional Euclidean space and let $u\in C^\infty(B)$.  Then
 \begin{multline*}
  \mE_4(u) = \int_B \left[ \frac{2}{3}T_1(u)(\nabla u,\nabla u) + \frac{n}{6}\lv\nabla u\rv^4 \right] \\ + \oint_{\partial B} \left[ \frac{n-3}{6}uH(u)\lv\onabla u\rv^2 + \frac{4}{3}\left(\frac{n-3}{4}\right)^2u^2\lv\onabla u\rv^2 + \frac{n}{3}\left(\frac{n-3}{4}\right)^3u^4 \right] .
 \end{multline*}
\end{lem}

\begin{proof}
 It follows from~\eqref{eqn:4Laplacian} that
 \begin{multline*}
  \int_B \left[ 2T_1(u)(\nabla u,\nabla u) - 3\lv\nabla u\rv^2\sigma_1(u) + \frac{n-3}{2}\lv\nabla u\rv^4 \right] \\ = \frac{n-3}{4}\oint_{\partial B} \left[uH(u)\lv\onabla u\rv^2 - \frac{n-3}{4}u^2\lv\onabla u\rv^2 + u(\eta u)^3\right] .
 \end{multline*}
 Using this to eliminate the term $\oint u(\eta u)^3$ from~\eqref{eqn:mE4_positive} yields the desired conclusion.
\end{proof}

We now have the ingredients in place to classify local minimizers of $\mE_2\colon\mV\to\bR$ in dimension $n+1\leq 6$.

\begin{proof}[Proof of Theorem~\ref{thm:main_thm}]
 Let $u$ be a local minimizer of $\mE_2\colon\mV\to\bR$. 
 Proposition~\ref{prop:balanced} implies that, by using the action of $\Conf(B^{n+1};S^n)$ if necessary, we may assume that $u$ satisfies~\eqref{eqn:balanced}.  On the one hand, Corollary~\ref{cor:balanced_minimizer} states that
 \[ \sum_{i=1}^{n+1}\biggl\{ \int_B ux^i\,[L_4,x^i](u,u,u) + \oint_{\partial B} ux^i\,[B_3,x^i](u,u,u)\biggr\} \geq \frac{4}{n-3}\mE_2(u) . \]
 On the other hand, Lemma~\ref{lem:L4_commutator} and Lemma~\ref{lem:B3_commutator} imply that
 \begin{multline*}
  \sum_{i=1}^{n+1}\biggl\{ \int_B ux^i\,[L_4,x^i](u,u,u) + \oint_{\partial B} ux^i\,[B_3,x^i](u,u,u)\biggr\} \\ = \frac{4}{3(n-3)}\int_B u\left(T_1(u)-\frac{n}{2}\sigma_1(u)g\right)(\nabla u,\nabla r^2) \\ + \frac{1}{3}\oint_{\partial B} \Bigl[ u^2T_1(u)(\eta,\eta) + \frac{n(n-3)}{4}u^3H(u)\Bigr]
 \end{multline*}
 Combining these displays using~\eqref{eqn:int_T1eta} and Lemma~\ref{lem:alternate_energy} yields
 \begin{multline*}
  \frac{4}{3(n-3)}\int_B \left[ u\left(T_1(u)-\frac{n}{2}\sigma_1(u)g\right)(\nabla u,\nabla r^2) - \frac{n}{2}\lv\nabla u\rv^4 - 2T_1(u)(\nabla u,\nabla u) \right] \\ \geq \oint_{\partial B} \left[ \frac{2}{3}uH(u)\lv\onabla u\rv^2 + \frac{n-2}{6}u^2\lv\onabla u\rv^2 + \frac{n}{6}u^2(\eta u)^2 \right] .
 \end{multline*}
 Combining this with Lemma~\ref{lem:integral_nablar} yields
 \begin{multline}
  \label{eqn:final_stability_estimate}
  0 \geq \int_B \left[ \frac{2n}{3(n-3)}\lv\nabla u\rv^4 + \frac{8}{3(n-3)}T_1(u)(\nabla u,\nabla u) - \frac{n(n-5)}{6}u^2\lv\nabla u\rv^2 \right] \\ + \oint_{\partial B} \left[ \frac{2}{3}uH(u)\lv\onabla u\rv^2 + \frac{n-3}{3}u^2\lv\onabla u\rv^2 \right] .
 \end{multline}
 Since $n\leq5$, we see that the right-hand side is nonnegative, and hence equality holds in~\eqref{eqn:final_stability_estimate}.  Therefore $u$ is constant.
\end{proof}

\subsection{The functional $\mG_2$}
\label{subsec:classification/critical_k2}

We conclude by considering local minimizers $u\in\mC_1$ of the functional $\mG_2\colon C^\infty(B^4)\to\bR$. Our first task is to compute $L_{4,j}(x,u,\dotsc,u)$ and $B_{3,j}(x,u,\dotsc,u)$.

\begin{lem}
 \label{lem:critical_Lx}
 Let $(B^4,dx^2)$ be the unit ball in Euclidean four-space and let $x$ denote a Cartesian coordinate in $\bR^4$.  Then
 \begin{align}
  \label{eqn:L43x} L_{4,3}(x,u,u) & = 2\lp\nabla x,\nabla u\rp\,\Delta u + 4\nabla^2u(\nabla u,\nabla x) , \\
  \label{eqn:L42x} L_{4,2}(x,u) & = 0
 \end{align}
 for all $u\in C^\infty(B)$.
\end{lem}

\begin{proof}
 By direct computation,
 \[ L_{4,3}(x,u,u) = 2\delta\left(\lp\nabla x,\nabla u\rp\,du\right) + \delta\left(\lv\nabla u\rv^2\,dx\right) . \]
 Expanding this using the fact $\nabla^2x=0$ yields~\eqref{eqn:L43x}.  By direct computation again,
 \[ L_{4,2}(x,u) = -\frac{1}{2}\left[ \Delta\lp\nabla x,\nabla u\rp - \delta\left((\Delta u)\,dx\right) \right] . \]
 We deduce~\eqref{eqn:L42x} from the facts that $dx^2$ is flat and $\nabla^2x=0$.
\end{proof}

\begin{lem}
 \label{lem:critical_Bx}
 Let $(B^4,dx^2)$ be the unit ball in Euclidean four-space and let $x$ denote a Cartesian coordinate in $\bR^4$.  Then
 \begin{align}
  \label{eqn:B33x} B_{3,3}(x,u,u) & = -x\left(\lv\onabla u\rv^2+3(\eta u)^2\right) - 2\lp\onabla x,\onabla u\rp\,\eta u , \\
  \label{eqn:B32x} B_{3,2}(x,u) & = -x\oDelta u - \lp\onabla u,\onabla x\rp , \\
  \label{eqn:B31x} B_{3,1}(x) & = 3x
 \end{align}
 for all $u\in C^\infty(B)$.
\end{lem}

\begin{proof}
 Recall that $-\oDelta x=3x$.  The conclusion follow by direct computation.
\end{proof}

We obtain the following analogue of Lemma~\ref{lem:integral_nablar}.
Define the $\Gamma(\otimes^2T^\ast S^3)$-valued differential operator $T_1$ by
\[ T_1(u) := \nabla^2u - du\otimes du - \left(\Delta u + \frac{1}{2}\lv\nabla u\rv^2\right)\,dx^2 . \]
for all $u\in C^\infty(B)$.  Note that $T_1(u)=T_1^{e^{2u}dx^2}$, so that $T_1(u)\geq0$ if $e^{2u}dx^2\in\overline{\Gamma_2^+}$.

\begin{lem}
 \label{lem:critical_integral_nablar}
 Let $(B^4,dx^2)$ be the unit ball in Euclidean four-space and let $r^2\in C^\infty(B)$ denote the squared-distance from the origin.  Then
 \[ \int_{B^4} \left( T_1(u) - \frac{3}{2}\sigma_1(u)dx^2\right)(\nabla u,\nabla r^2) = -3\int_{B^4}\sigma_1(u) - \oint_{S^3} T_1(u)(\eta,\eta) . \]
\end{lem}

\begin{proof}
 First observe that
 \[ \delta T_1(u) = -(\Delta u)\,du - d\lv\nabla u\rv^2 . \]
 Therefore
 \begin{equation}
  \label{eqn:critical_divergence_identity}
  \delta\left(T_1(u)(\nabla r^2)\right) = -6\sigma_1(u) - \lp\nabla r^2,\nabla u\rp\,\Delta u - 2\nabla^2u(\nabla u,\nabla r^2) .
 \end{equation}
 Second observe that
 \begin{equation}
  \label{eqn:critical-T1sigma1}
  \left(T_1(u)-\frac{3}{2}\sigma_1(u)g\right)(\nabla u,\nabla r^2) = \nabla^2u(\nabla u,\nabla r^2) + \frac{1}{2}\lp\nabla u,\nabla r^2\rp\,\Delta u .
 \end{equation}
 Combining these results with the Divergence Theorem yields the desired conclusion.
\end{proof}

We now can classify local minimizers of $\mG_2\colon C^\infty(B^4)\to\bR$ which are in $\mC_1$.

\begin{proof}[Proof of Theorem~\ref{thm:main_thm3}]
 Let $u\in \mC_1$ be a local minimizer of $\mG_2\colon C^\infty(B)\to\bR$.  Proposition~\ref{prop:critical_balanced} implies that, by using the action of $\Conf(B^4;S^3)$ if necessary, we may assume that $u$ satisfies~\eqref{eqn:critical_balanced}.  Combining Corollary~\ref{cor:critical_balanced_minimizer}, Lemma~\ref{lem:critical_Lx}, Lemma~\ref{lem:critical_Bx} and~\eqref{eqn:critical-T1sigma1} yields
 \[ \int_{B^4} \left(T_1(u) - \frac{3}{2}\sigma_1(u)g\right)(\nabla u,\nabla r^2) \geq \frac{1}{2}\oint_{S^3} \Bigl[ \lv\onabla u\rv^2 + 3(\eta u)^2 \Bigr] . \]
 Combining this with Lemma~\ref{lem:critical_integral_nablar} yields
 \[ 0 \geq 3\int_{B^4} \sigma_1(u) + \oint_{S^3} \left[ T_1(u)(\eta,\eta) + \frac{1}{2} \lv\onabla u\rv^2 + \frac{3}{2}(\eta u)^2 \right] \geq 0 . \]
 Therefore equality holds in both steps.  In particular, $\eta u=0$ and
 \[ 0 = \int_{B^4}\sigma_1(u) = -\int_{B^4} \lv\nabla u\rv^2 . \]
 We conclude that $u$ is constant, as desired.
\end{proof}

\bibliographystyle{abbrv}
\bibliography{bib}
\end{document}